\documentclass{article}

\usepackage{a4}
\usepackage{paralist}

\title{Robust discretization and solvers for elliptic optimal control
  problems with energy regularization} 
\author{Ulrich~Langer\footnote{Institute for Computational Mathematics,
    Johannes Kepler University Linz, Altenberger Stra{\ss}e 69, 4040 Linz, Austria, Email: ulanger@numa.uni-linz..ac.at},
    \; Olaf~Steinbach\footnote{Institut f\"{u}r Angewandte Mathematik,
    Technische Universit\"{a}t Graz, Steyrerga{\ss}e 30, 8010 Graz, Austria,
    Email: o.steinbach@tugraz.at}, 
    \; Huidong~Yang\footnote{Johann Radon Institute for Computational and Applied
    Mathematics, Austrian Academy of Sciences, Altenberger Stra{\ss}e 69, 4040 Linz, Austria, Email:
    huidong.yang@ricam.oeaw.ac.at} 
}  

\date{}

\usepackage{amsmath}
\usepackage{amsthm}
\usepackage{amsfonts}
\usepackage{booktabs}
\usepackage{graphicx}
\usepackage{diagbox}

\usepackage{multirow}
\usepackage{hyperref}

\usepackage[ulem=normalem,draft]{changes}

\newcommand{\be}{\begin{equation}}
\newcommand{\ee}{\end{equation}}

\newcommand{\ignore}[1]{}

\newtheorem{theorem}{Theorem}

\begin{document}

\maketitle

\begin{abstract}
  We consider the finite element discretization and the iterative
  solution of singularly perturbed elliptic reaction-diffusion equations
  in three-dimensional computational domains. These equations arise from
  the optimality conditions for elliptic distributed optimal control
  problems with energy regularization that were recently studied by
  M.~Neum\"{u}ller and O.~Steinbach (2020). We provide quasi-optimal
  a priori finite element error estimates which depend both on the
  mesh size $h$ and on the regularization parameter $\varrho$. The 
  choice $\varrho = h^2$ ensures optimal convergence which only depends
  on the regularity of the target function. For the iterative solution,
  we employ an algebraic multigrid preconditioner and a
  balancing domain decomposition by constraints (BDDC) preconditioner.
  We numerically study robustness and efficiency of the proposed algebraic
  preconditioners with respect to the mesh size $h$, the regularization
  parameter $\varrho$, and the number of subdomains (cores) $p$.
  Furthermore, we investigate the parallel performance of the BDDC
  preconditioned conjugate gradient solver.\\
\end{abstract} 

\begin{keywords}
Elliptic optimal control problems, energy regularization,\\
     \hspace*{24mm} finite element discretization,
 a priori error estimates, fast solvers \\
\end{keywords}

\begin{msc}
49K20, 49M41,  35J25, 65N12, 65N30, 65N55
\end{msc}

%
%
\section{Introduction}\label{sec:intro}
In the recent work \cite{NeuSte20}, M.~Neum\"{u}ller and O.~Steinbach have
investigated regularization error estimates for the solution of distributed
optimal control problems in energy spaces on the basis of the following
model problem: Minimize the cost functional
\begin{equation}\label{eq:optcon}
  {\mathcal J}(u_\varrho, z_\varrho) \, = \,
  \frac{1}{2} \int_\Omega [ u_\varrho(x)-\overline{u} (x) ]^2 \, dx +
  \frac{\varrho}{2} \, \|z_\varrho\|_{H^{-1}(\Omega)}^2
\end{equation} 
with respect to the state $u_\varrho \in H^1_0(\Omega)$ and the control
$z_\varrho\in H^{-1}(\Omega)$ subject to the constraints
\begin{equation}\label{eq:state}
  -\Delta u_\varrho = z_\varrho \quad \textup{in} \;  \Omega,
  \quad u_\varrho=0 \quad \textup{on}\; \Gamma,
\end{equation}
where $\overline{u} \in L^2(\Omega)$ represents the given desired state
(target), and $\varrho\in {\mathbb R}_+$ denotes the regularization
parameter. Here, the computational domain $\Omega\subset{\mathbb R}^3$ is
assumed to be a bounded Lipschitz domain with boundary $\Gamma=\partial\Omega$.
Note that
\[
\|z_\varrho\|_{H^{-1}(\Omega)} = \| \nabla u_\varrho \|_{L^2(\Omega)} .
\]
The associated adjoint state equation reads: Find the adjoint state
$p_\varrho\in H^1_0(\Omega)$ such that 
\begin{equation}\label{eq:adjoint}
  -\Delta p_\varrho  = u_\varrho - \overline{u}
   \quad \textup{in} \;  \Omega,
  \quad  p_\varrho=0 \quad \textup{on} \; \Gamma.
\end{equation}
Further, the associated gradient equation is given by the relation
\begin{equation}\label{eq:grad}
  p_\varrho + \varrho \, u_\varrho = 0  \quad \textup{ in }  \Omega.
\end{equation}
Using the gradient equation (\ref{eq:grad}) for eliminating the adjoint
state $p_\varrho$ from (\ref{eq:adjoint}), we finally obtain the singularly
perturbed reaction-diffusion equation
\begin{equation}\label{eq:sigular}
  -\varrho \Delta u_\varrho  + u_\varrho  = \overline{u}  \quad
  \textup{in} \;  \Omega, \quad u_\varrho=0 \quad \textup{on}\; \Gamma,
\end{equation} 
for defining the state $u_\varrho$.
Once the state $u_\varrho$ is computed from \eqref{eq:sigular}, the adjoint
state $p_\varrho$ is given by \eqref{eq:grad}, and the control
$z_\varrho = -\Delta u_\varrho = \varrho^{-1} (\overline{u} - u_\varrho)$
by \eqref{eq:state} and \eqref{eq:sigular}.

The variational formulation of the Dirichlet boundary value problem
(\ref{eq:sigular}) is to find $u_\varrho \in H_0^1(\Omega)$ such that
\begin{equation}\label{eq:varform}
  \varrho \int_\Omega \nabla u_\varrho(x) \cdot \nabla v(x) \, dx +
  \int_\Omega u_\varrho (x) \, v(x) \, dx =
  \int_\Omega \overline{u}(x) \, v(x) \, dx 
\end{equation}
is satisfied for all $v\in H_0^1(\Omega)$. When assuming some regularity of
the given target
$\overline{u}\in H_0^s(\Omega):=[L^2(\Omega), H_0^1(\Omega)]_s$ for
$s\in[0,1]$ or $\overline{u} \in H^1_0(\Omega) \cap H^s(\Omega)$ for
$s \in (1,2]$, the following error estimate with respect to the regularization
parameter $\varrho$ has been shown in \cite{NeuSte20}:
\begin{equation}\label{eq:errest} 
  \| u_\varrho - \overline{u} \|_{L^2(\Omega)} \, \leq \, 
  c \, \varrho^{s/2} \, \|\overline{u}\|_{H^s(\Omega)}.
\end{equation}
Moreover, in the case $\overline{u} \in H^1_0(\Omega) \cap H^s(\Omega)$ for $s \in (1,2]$, 
there also holds
\begin{equation}\label{eq:errest H1} 
  \| \nabla (u_\varrho - \overline{u}) \|_{L^2(\Omega)} \, \leq \, 
  c \, \varrho^{(s-1)/2} \, \|\overline{u}\|_{H^s(\Omega)}.
\end{equation}
For $\varrho \to 0$, the boundary value problem (\ref{eq:sigular})
belongs to a class of singularly perturbed problems, see, e.g.,
\cite{Tobiska1983}. For robust numerical methods to treat such problems,
we refer to the monograph \cite{Tobiska2008}.

In this paper, we investigate the errors
$\|u_{\varrho h} - \overline{u} \|_{L^2(\Omega)}$  
and $\|\nabla(u_{\varrho h} - \bar{u}) \|_{L^2(\Omega)}$
of the finite element approximation $u_{\varrho h}$ to the target $\overline{u}$ 
in terms of both the regularization parameter  $\varrho$ and the discretization 
parameter $h$. The error estimate is based on the estimates \eqref{eq:errest}
and \eqref{eq:errest H1},
and well-known finite element discretization error estimates.
It turns out that the choice $\varrho = h^2$ gives 
the optimal convergence rate $O(h^2)$ for continuous, piecewise linear 
finite elements. Moreover,  we consider iterative methods for the solution
of the linear system of algebraic equations that arise from the Galerkin
finite element discretization of the variational problem  (\ref{eq:varform}),
which are robust with respect to the mesh size $h$ and the
regularization parameter $\varrho$. Such solvers have been addressed in many
works. For example, geometric multigrid methods, 
which are robust with respect to the mesh size $h$ and the parameter $\varrho$ 
in the case of standard Galerkin discretization methods, 
were analyzed in \cite{OlsReu00}, based on the
convergence rate $\min\{1, h^2/\varrho\}$ in the $L^2$ norm as provided
in \cite{SchaWahl83}. The parameter independent
contraction number of the multigrid method is achieved by a combination of
such a deteriorate approximation property with an improved smoothing
property. Robust and optimal algebraic multilevel (AMLI) methods for
reaction-diffusion type problems were studied in \cite{KraWolf13}. Therein, a
uniformly convergent AMLI method with optimal complexity was shown using the
constant in the strengthened Cauchy-Bunyakowski-Schwarz inequality computed
for both mass and stiffness matrices. Robust block-structured
preconditioning on both piecewise uniform meshes and graded meshes with
boundary layers has been considered in \cite{MacMad13,NhanMacMad18} in
comparison with standard robust multigrid preconditioners. The new block
diagonal preconditioner was based on the partitioning of degrees of freedom
into those on the corners, edges, and interior points,
respectively. The perturbation parameter independent condition
number of the preconditioned linear system using diagonal and incomplete
Cholesky preconditioning methods for singularly perturbed problems on
layer-adapted meshes has been recently shown in \cite{NHANMADDEN2021}.
Considering the adaptive finite element discretization on
simplicial meshes with local refinements, we employ an
algebraic multigrid (AMG) preconditioner
\cite{Brandt1985,Briggs2000,Ruge1987} and the balancing domain
decomposition by constraints (BDDC) preconditioner
\cite{Dohrmann2003,MandelDohrman2003,Mandel2005} for solving the
discrete system. We make  performance studies 
of the preconditioned conjugate gradient (PCG) method 
that uses these AMG and BDDC preconditioners. 
In particular, the BDDC preconditioned conjugate gradient solver
shows excellent strong scalability.

The reminder of the paper is organized as follows: Section
\ref{sec:discretization} deals with the finite element discretization of the
equation (\ref{eq:sigular}). In Section \ref{sec:preconditioners}, we
describe both the AMG and BDDC preconditioners that are used 
in solving the system of finite element equations.
Numerical results are presented and discussed in Section
\ref{sec:nume}. Finally, some conclusions are drawn in Section \ref{sec:con}.  

%
%
\section{Discretization}\label{sec:discretization}
To perform the finite element discretization of the variational form
\eqref{eq:varform}, we introduce conforming finite element spaces
$V_h \subset H_0^1(\Omega)$. Particularly, we use the standard finite
element space $V_h = S_h^1(\Omega_h) \cap H_0^1(\Omega)$ spanned by
continuous and piecewise linear basis functions. These functions are
defined with respect to some admissible decomposition
$\mathcal{T}_h(\Omega)$ of the domain $\Omega$ into shape
regular simplicial finite elements $\tau_\ell$, and are zero on
$\Gamma$. Here, $\overline{\Omega}_h = \bigcup_\ell \overline{\tau}_\ell$,
and $h$ denotes a suitable mesh-size parameter, see, e.g., 
\cite{Braess2007,Steinbach:2008a}. Then the finite element approximation
of \eqref{eq:varform} is to find $u_{\varrho h} \in V_h$ such that
\begin{equation}\label{eqn:distvarform}
  \varrho\int_\Omega \nabla u_{\varrho h}(x )\cdot\nabla v_h(x) \, dx +
  \int_\Omega u_{\varrho h}(x) \, v_h(x) \, dx =
  \int_\Omega \overline{u}(x) \, v_h (x) \, dx 
\end{equation}
for all test functions $v_h\in V_h$.

In \cite{NeuSte20}, the convergence of $u_\varrho$ towards the target
$\overline{u}$, and the error estimates \eqref{eq:errest} and
\eqref{eq:errest H1} were proved.
In the numerical experiments, presented in \cite{NeuSte20}, $u_\varrho$ 
was computed on a very fine grid to minimize the influence of the
discretization. Now we are going to investigate the effects of the
finite element discretization, and the final aim is to provide estimates
of the errors $\| u_{\varrho h} - \overline{u} \|_{L^2(\Omega)}$ and
$\| \nabla (u_{\varrho h} - \overline{u}) \|_{L^2(\Omega)}$
in terms of both $\varrho$ and $h$ for sufficiently smooth target
functions $\overline{u}$.

\begin{theorem} \label{thm:errest}
  Let us assume that $\Omega \subset {\mathbb{R}}^3$ is convex, and that the
  target function satisfies $\overline{u} \in H^1_0(\Omega) \cap H^2(\Omega)$.
  Then there are positive constants $C_0$, $C_1$, $C_2$ and $C_3$ such that the
  error estimates
  \begin{equation}\label{eqn:error}
    \| u_{\varrho h} - \overline{u} \|_{L^2(\Omega)}^2 \leq
    \Big( C_0 h^4 + C_1 \varrho h^2 + C_2  \varrho^2 \Big) \,
    \| \overline{u} \|_{H^2(\Omega)}^2
  \end{equation}
  and
  \begin{equation}\label{eqn:error H1}
    \| \nabla (u_{\varrho h} - \overline{u}) \|_{L^2(\Omega)}^2 \leq
    \Big( C_0 h^4 \varrho^{-1} + C_1 h^2 + C_3  \varrho \Big) \,
    \| \overline{u} \|_{H^2(\Omega)}^2
  \end{equation}
  hold. For the choice $\varrho = h^2$, we therefore have 
  \begin{equation*}
    \| u_{\varrho h} - \overline{u} \|_{L^2(\Omega)} \leq
    c \, h^2 \, \| \overline{u} \|_{H^2(\Omega)} 
    \quad \mbox{and} \quad 
    \| \nabla (u_{\varrho h} - \overline{u}) \|_{L^2(\Omega)} \leq
    \tilde{c} \, h \, \| \overline{u} \|_{H^2(\Omega)}.
  \end{equation*}
\end{theorem}

\begin{proof}
  When subtracting the Galerkin variational formulation 
  \eqref{eqn:distvarform} from \eqref{eq:varform} for test functions
  $v_h \in V_h$ this gives the Galerkin orthogonality
  \[
    \varrho \int_\Omega \nabla (u_\varrho - u_{\varrho h}) \cdot \nabla v_h \, dx
    + \int_\Omega (u_\varrho - u_{\varrho h}) \, v_h \, dx = 0
    \quad \forall v_h \in V_h.
  \]
  Therefore, 
  \begin{eqnarray*}
    && \varrho \, \| \nabla(u_\varrho-u_{\varrho h})\|^2_{L^2(\Omega)} +
       \| u_\varrho-u_{\varrho h} \|^2_{L^2(\Omega)} \\
    && \hspace*{2mm}
       = \varrho \int_\Omega \nabla (u_\varrho - u_{\varrho h}) \cdot
       \nabla (u_\varrho - u_{\varrho h}) \, dx + \int_\Omega
       (u_\varrho - u_{\varrho h}) \, (u_\varrho - u_{\varrho h}) \, dx\\
    && \hspace*{2mm}
       = \varrho \int_\Omega \nabla (u_\varrho - u_{\varrho h}) \cdot
       \nabla (u_\varrho - v_h) \, dx + \int_\Omega
       (u_\varrho - u_{\varrho h}) \, (u_\varrho - v_h) \, dx \\
    && \hspace*{2mm}
       \leq \, \varrho \, \| \nabla (u_\varrho - u_{\varrho h}) \|_{L^2(\Omega)}
       \| \nabla (u_\varrho - v_h)\|_{L^2(\Omega)} +
       \| u_\varrho - u_{\varrho h} \|_{L^2(\Omega)}
       \| u_\varrho - v_h \|_{L^2(\Omega)} \\
    && \hspace*{2mm}
       \leq \sqrt{\varrho \, \| \nabla(u_\varrho-u_{\varrho h})\|^2_{L^2(\Omega)}
       + \| u_\varrho-u_{\varrho,h} \|^2_{L^2(\Omega)}} \\
    && \hspace*{20mm}
       \cdot \sqrt{ \varrho \, \| \nabla(u_\varrho-v_h)\|^2_{L^2(\Omega)} +
       \| u_\varrho-v_h \|^2_{L^2(\Omega)}}
\end{eqnarray*}
follows, i.e., we have Cea's estimate
\[
  \varrho \, \| \nabla(u_\varrho-u_{\varrho h})\|^2_{L^2(\Omega)} +
  \| u_\varrho-u_{\varrho h} \|^2_{L^2(\Omega)} \leq
  \varrho \, \| \nabla(u_\varrho-v_h)\|^2_{L^2(\Omega)} +
  \| u_\varrho-v_h \|^2_{L^2(\Omega)} 
\]
for all $v_h \in V_h$. Inserting the Lagrangian interpolation
$v_h = I_h u_\varrho$ of $u_\varrho$, and using standard interpolation
error estimates, we arrive  at the error estimate
\begin{eqnarray}\label{eqn:diserror}\nonumber
  \varrho \, \| \nabla(u_\varrho-u_{\varrho h})\|^2_{L^2(\Omega)} +
  \| u_\varrho-u_{\varrho h} \|^2_{L^2(\Omega)}
  & \leq & c_1 \, \varrho \, h^2 \, |u_\varrho|^2_{H^2(\Omega)} +
           c_2 \, h^4 \, |u_\varrho|^2_{H^2(\Omega)} \\
  & = & \left[ c_1 \, \varrho + c_2 \, h^2 \right]
        \, h^2 \,  |u_\varrho|^2_{H^2(\Omega)},
\end{eqnarray}
where the positive constants $c_1$ and $c_2$ are nothing but the constants 
in the $H^1$ and $L^2$ interpolation error estimates; 
see, e.g., \cite{Braess2007,Steinbach:2008a}. Note that, under the
assumptions made, we have $u_\varrho \in H_0^1(\Omega) \cap H^2(\Omega)$.
Due to
\[
- \varrho \Delta u_\varrho = \overline{u} - u_\varrho,
\]
we now conclude
\begin{equation}\label{eqn:coercitivity}
  |u_\varrho|^2_{H^2(\Omega)} \leq
  c_3 \, \frac{1}{\varrho^2} \, \| \overline{u} - u_\varrho \|^2_{L^2(\Omega)}
\end{equation}
with the positive $H^2$ coercivity constant $c_3$.
Since we assume $\overline{u} \in H^1_0(\Omega) \cap H^2(\Omega)$,
we can use \eqref{eq:errest} for $s=2$, i.e.,
\begin{equation}\label{eqn:errestNeuSte20}
  \|\overline{u} - u_\varrho \|_{L^2(\Omega)} \leq
  c_4 \, \varrho \, \| \overline{u} \|_{H^2(\Omega)} .
\end{equation}
For less regular $\overline{u}$, we have a reduced order in $\varrho$;
see Theorem 3.2 in \cite{NeuSte20} as well as \eqref{eq:errest}.
Combining \eqref{eqn:coercitivity} and \eqref{eqn:errestNeuSte20}, 
we get
\begin{equation}\label{eqn:uvarrhoubar}
  |u_\varrho|^2_{H^2(\Omega)} \leq c_3 c_4^2 \, \| \overline{u} \|^2_{H^2(\Omega)}.
\end{equation}
Inserting \eqref{eqn:uvarrhoubar} into \eqref{eqn:diserror} this gives
\begin{equation*}
  \varrho \, \| \nabla(u_\varrho-u_{\varrho h})\|^2_{L^2(\Omega)} +
  \| u_\varrho-u_{\varrho h} \|^2_{L^2(\Omega)} \leq 
  c_3 c_4^2 \, \left[ c_1 \, \varrho + c_2 \, h^2 \right]
  \, h^2 \,  \| \overline{u} \|^2_{H^2(\Omega)}.
\end{equation*}
Using the triangle inequality, and again \eqref{eqn:errestNeuSte20},
we arrive at
\begin{eqnarray*}
  \| u_{\varrho h} - \overline{u} \|_{L^2(\Omega)}^2
  & \leq & \Big[
           \| u_{\varrho h} - u_\varrho \|_{L^2(\Omega)} +
           \| u_\varrho - \overline{u} \|_{L^2(\Omega)} \Big]^2 \\
  & \leq & 2 \, \Big[ \| u_{\varrho h} - u_\varrho \|^2_{L^2(\Omega)} +
           \| u_\varrho - \overline{u} \|^2_{L^2(\Omega)}
           \Big] \\
  & \leq & 2 \, c_3 c_4^2\,
           \left[ c_1 \, \varrho + c_2 \, h^2 \right]
           \, h^2 \,  \| \overline{u} \|^2_{H^2(\Omega)} +
           2 c_4^2 \, \varrho^2 \, \| \overline{u} \|^2_{H^2(\Omega)} ,
\end{eqnarray*}
that is nothing but \eqref{eqn:diserror} with
$C_0 = 2c_2c_3c_4^2$, $C_1 = 2c_1c_3c_4^2$, and $C_2 = 2 c_4^2$.

Moreover, we also have
\begin{equation*}
  \| \nabla(u_\varrho-u_{\varrho h})\|^2_{L^2(\Omega)} 
  \leq 
  c_3 c_4^2 \, \left[ c_1 \, \varrho + c_2 \, h^2 \right]
  \, h^2 \, \varrho^{-1} \, \| \overline{u} \|^2_{H^2(\Omega)}.
\end{equation*}
Now, using \eqref{eq:errest H1} for $s=2$, we finally obtain
  \begin{eqnarray*}
  \| \nabla (u_{\varrho h} - \overline{u}) \|_{L^2(\Omega)}^2
  & \leq & \Big[
           \| \nabla (u_{\varrho h} - u_\varrho) \|_{L^2(\Omega)} +
           \| \nabla (u_\varrho - \overline{u}) \|_{L^2(\Omega)} \Big]^2 \\
  & \leq & 2 \, \Big[ \| \nabla (u_{\varrho h} - u_\varrho) \|^2_{L^2(\Omega)} +
           \| \nabla (u_\varrho - \overline{u}) \|^2_{L^2(\Omega)}
           \Big] \\
  & \leq & 2 \, c_3 c_4^2\,
           \left[ c_1 \, \varrho + c_2 \, h^2 \right]
           \, h^2 \, \varrho^{-1} \, \| \overline{u} \|^2_{H^2(\Omega)} +
           2 c_5^2 \, \varrho \, \| \overline{u} \|^2_{H^2(\Omega)} ,
\end{eqnarray*}
which is \eqref{eqn:error H1} with $C_3 = 2 c_5^2$.
\end{proof}

\noindent
The finite element space $V_h = \mbox{span}\{\varphi_1,\ldots,\varphi_{N_h}\}$ 
is spanned by the standard nodal Courant basis functions
$\{\varphi_1,\ldots,\varphi_{N_h}\}$. Once the basis is chosen, the finite
element scheme \eqref{eqn:distvarform} is equivalent to the system of
finite element equations
\begin{equation}\label{eq:linsys}
  A \underline{u} = \underline{f},
\end{equation}
where the system matrix $A=\varrho K + M$ consists of the scaled stiffness
matrix $K$ and the mass matrix $M$, respectively. The matrix entries and the
entries of the right-hand side $\underline{f}$ are defined by
\[
K_{ij} = \int_\Omega \nabla \varphi_j \cdot \nabla \varphi_i \, dx, \quad 
M_{ij} = \int_\Omega  \varphi_j \, \varphi_i \,dx, \quad
f_i = \int_\Omega  \overline{u} \, \varphi_i\,dx
\]
for $i,j=1,\ldots,N_h$. The solution
$\underline{u} = (u_j) \in \mathbb{R}^{N_h}$ of \eqref{eq:linsys}
contains the unknown coefficients of the finite element solution
\[
  u_{\varrho h} = \sum_{j=1}^{N_h} u_j \varphi_j \in V_h
\]
of \eqref{eqn:distvarform}. Robust and fast iterative solvers for
\eqref{eq:linsys} will be considered in the next section.

%
%
\section{AMG and BDDC Preconditioners}\label{sec:preconditioners}
In this section, we present two precondioners that lead to robust and
efficient solvers for the system \eqref{eq:linsys} of finite element equations
when applying a preconditioned conjugate gradient algorithm. 
The BDDC preconditioner is especially suited for parallel computers.

\subsection{The AMG preconditioner}\label{sec:amg}
In constrast to geometric multigrid methods, e.g, \cite{WHackbusch1985},
usually relying on underlying hierarchical meshes, algebraic multigrid 
methods \cite{Brandt1985,Briggs2000,Ruge1987,XUZikatanov2017} are more
flexible with respect to complex geometries, adaptive mesh refinements,
and so on. We refer to \cite{HaaseLanger2002} for a comparison between
these two methods. In this work, the particular AMG
preconditioner developed in \cite{KF1998} in combination with the 
conjugate gradient (CG) method is adopted to solve the discrete 
system. This AMG method was originally developed for second-order
elliptic problems. It has been extended to act as a robust preconditioner to 
systems arising from the discretization of coupled vector field problems,
e.g., to fluid and elasticity problems in fluid-structure interaction solvers 
\cite{Langer2016,YangZulehner2011}. More recently, in \cite{Steinbach2018}, 
by choosing a proper smoother \cite{Popa2008}, it was also used as a
preconditioner for the non-symmetric and positive definite system that
arises from the Petrov-Galerkin space-time finite element discretization
of the heat equation \cite{OS15}.  

In this AMG method, the coarsening strategy is based on a simple red-black
colouring algorithm \cite[Algorithm 2]{KF1998}: Pick up an unmarked degree
of freedom (Dof), and mark it black (coarse); all neighboring Dofs are marked
in red (fine); repeat this procedure until all Dofs are visited. Afterwards,
the fine Dofs (red) are interpolated by an average over their neighboring
coarse Dofs (black), which defines a linear interpolation operator $P$.
The coarse grid operator is constructed by Galerkin's method, i.e.,
$A_c=P^\top A P$. Assume that we apply $m$ pre- and post-smoothing steps, the
iteration operator for the two grid method is given by
\[
  M_{amg} := {\mathcal S}^m ( I -P A_c^{-1}P^\top A){\mathcal S}^m,
\]
where ${\mathcal S}$ is the iteration matrix for a smoothing step. For such
a symmetric and positive definite system, one sweep of the symmetric
Gauss–Seidel method is used, that is, one forward on the
down cycle together with one backward on the up cycle. 
Then, we may formulate the two-grid
preconditioner $P_{AMG}$ as follows (with $m=1$):  
\begin{equation}\label{eq:amgpre}
  P_{AMG}^{-1}:={\mathcal S} P A_c^{-1} P^\top {\mathcal S} .
\end{equation}
Note that, in the coarse grid correction step, one may apply AMG
recursively; see \cite{Brandt1985,Briggs2000,Ruge1987}, i.e.,
we replace $A_c^{-1}$ by AMG iterations starting with a zero initial guess.  
In connection with optimal control problems as considered in this paper,
we are especially interested in small $\varrho\in (0, 1]$. 
When $\varrho\to 0$, the mass matrix $M$ becomes dominant in $A$.
For the mass matrix, the condition number is ${\mathcal O}(1)$, which even
makes the system easier to solve as observed from our numerical tests.  

\subsection{The BDDC preconditioner}\label{sec:bddc}
For the (two-level) BDDC preconditioner, we first make a reordering of the
system (\ref{eq:linsys}), i.e.,
\begin{equation}\label{eq:dmpsys}
  \tilde{A}\tilde{\underline{u}}=\tilde{f},
\end{equation}
where
\begin{equation*}
\tilde{A}:=
\begin{bmatrix}
  \tilde{A}_{II} & \tilde{A}_{IC} \\
  \tilde{A}_{C I} & \tilde{A}_{CC} 
\end{bmatrix},
\; \tilde{\underline{u}}:=
\begin{bmatrix}
  \tilde{\underline{u}}_I\\
  \tilde{\underline{u}}_C 
\end{bmatrix},
\; \tilde{f}:=\begin{bmatrix}
  \tilde{f}_I\\
  \tilde{f}_C
\end{bmatrix}, 
\; \mbox{and}\,\,
\tilde{A}_{II}=\text{diag}\left[\tilde{A}_{II}^1,...,\tilde{A}_{II}^{p}\right].
\end{equation*}
Here, $p$ represents the number of polyhedral subdomains $\Omega_i$, as well
as the number of cores. These subdomains $\Omega_i$ are obtained from the
graph partitioning tool \cite{Karypis1998} and 
a non-overlapping domain decomposition of $\Omega$ \cite{Widlund2005}. 
Following common notations, the
degrees of freedom in (\ref{eq:dmpsys}) are respectively decomposed into 
internal ($I$) and interface ($C$) Dofs.
Then, we arrive at the following Schur complement system which is defined
on the interface $\Gamma_C$:
\begin{equation}\label{eq:schur}
  S_C \underline{\tilde{u}}_C = \underline{g},
\end{equation}
where
\[
  S_C := \tilde{A}_{CC} - \tilde{A}_{C I}\tilde{A}_{II}^{-1}\tilde{A}_{IC},
  \quad
  \underline{g} := \underline{\tilde{f}}_C - \tilde{A}_{C I}
  \tilde{A}_{II}^{-1}\underline{\tilde{f}}_I.
\]
Now, following similar notations as in \cite{Dohrmann2003,Jing2006,
  MandelDohrman2003,Mandel2005}, the BDDC preconditioner $P_{BDDC}$ for
(\ref{eq:schur}) is formed as
\[
  P_{BDDC}^{-1}=R_C^\top(T_{sub}+T_0)R_C,
\] 
where the operator $R_C$ represents the direct sum of
restriction operators $R_C^i$ that map the global interface vector
to its components on a local interface
$\Gamma_i:=\partial \Omega_i\cap \Gamma_C$ with a proper scaling.
Furthermore, the coarse level correction operator $T_0$ is defined by
\begin{equation}\label{eq:coarsecorr}
  T_0 = \Phi (\Phi^\top S_C \Phi)^{-1} \Phi^\top,
\end{equation}
where $\Phi=\left[(\Phi^1)^\top,\ldots, (\Phi^N)^\top\right]^\top$ is 
the matrix of the coarse level basis functions.
Each basis function matrix $\Phi^i$ on the subdomain
interface $\Gamma_i$ is computed by solving the augmented system:
\begin{equation}
  \begin{aligned}
    \begin{bmatrix}
      S^i_C & \left(C^i\right)^\top \\
      C^i & 0
    \end{bmatrix}
    \begin{bmatrix}
      \Phi^i\\
      \Lambda^i\\
    \end{bmatrix}
    =
    \begin{bmatrix}
      0\\
      R_{\Pi}^i\\
    \end{bmatrix}
  \end{aligned},
\end{equation}
where $S^i_C$ is nothing but the local Schur complement, $C^i$ denotes the
given primal constraints of the subdomain $\Omega_i$, and each column of
$\Lambda^i$ contains the vector of Lagrange multipliers. The number
of columns of each $\Phi^i$ is equal to the number of global coarse level
degrees of freedom, usually living on the subdomain corners, and/or
interface edges, and/or faces. Moreover, the restriction operator $R_{\Pi}^i$
maps the global interface vector in the continuous primal variable space on
the coarse level to its component on $\Gamma_i$. We mention that these
corners/edges/faces on the coarse space may be characterized as objects from a
pure algebraic manner; see, e.g., \cite{Badia2013}. We have recently used this
abstract method in \cite{Langer2020} for constructing BDDC preconditioners in
solving the linear system that arises from the space-time finite element
discretization of parabolic equations \cite{Langer2019}.  

The subdomain correction operator $T_{sub}$ is then defined as 
\[
  T_{sub}=\displaystyle\sum_{i=1}^N\left[(R_C^i)^\top \quad 0\right]
  \begin{bmatrix}
     S_C^i & \left(C^i\right)^\top \\
     C^i & 0
  \end{bmatrix}^{-1}
  \begin{bmatrix}
     R_C^i\\
    0
  \end{bmatrix},
\]
with vanishing primal variables on all the coarse levels. Here, the restriction
operator $R_C^i$ maps global interface vectors to their components on
$\Gamma_i$. Note that this two-level method can be extended to a multi-level
method when the coarse problem becomes too large; see, e.g.,
\cite{Badia2016,ZampiniTu2017}. This means that we may apply the BDDC method
to the inversion of $\Phi^\top S_C \Phi$
in the coarse problem \eqref{eq:coarsecorr} recursively.

%
%
\section{Numerical experiments}\label{sec:nume}
In all numerical examples presented in this section, we consider the 
unit cube $\Omega=(0, 1)^3$ as computational domain.
In the first example, we choose a smooth target, whereas a discontinuous 
target is used in the second example. The first example is covered 
by Theorem~\ref{thm:errest}, the second not. Finally, we consider an example
with a smooth target that violates the homogeneous Dirichlet boundary
conditions. The system \eqref{eq:linsys} of finite element equations is solved 
by a preconditioned conjugate gradient (PCG) method with both AMG and BDDC
preconditioners. The PCG iteration is stopped when the relative residual
error of the preconditioned system reaches $\varepsilon=10^{-8}$. In the
AMG method, we have applied one $V$-cycle multigrid preconditioner with one
pre- and post-smoothing step, respectively. We run the AMG preconditioned CG
on the shared memory supercomputer
MACH-2\footnote{https://www3.risc.jku.at/projects/mach2/}. 
The BDDC preconditioned CG is performed on the high-performance 
distributed memory computing cluster
RADON-1\footnote{https://www.ricam.oeaw.ac.at/hpc/}.
    
\subsection{Smooth target}\label{subsec:SmoothTarget}
In the first example, we consider the smooth function
\[
  \bar{u} (x_1,x_2,x_3) =
  - \varrho \Delta u_\varrho + u_\varrho =
  \sin(\pi x_1 )\sin(\pi x_2)\sin(\pi x_3)
\]
as target, which was computed from the manufactured solution 
\[
  u_\varrho(x_1,x_2,x_3) =
  (3\varrho\pi^2+1)^{-1}\sin(\pi x_1 )\sin(\pi x_2)\sin(\pi x_3)
\]
of the reduced optimality equation (\ref{eq:sigular}).
Then the adjoint state 
\[
  p_\varrho(x_1,x_2,x_3) =
  - \frac{u_\varrho(x_1,x_2,x_3)}{\varrho (3\varrho\pi^2+1)}  =
  - \frac{1}{\varrho(3\varrho\pi^2+1)} \sin(\pi x_1 )
  \sin(\pi x_2)\sin(\pi x_3)
\]
results from the gradient equation, and the optimal control 
\[
  z_\varrho(x_1,x_2,x_3) = -\Delta u_\varrho(x_1,x_2,x_3)
  = \frac{3\pi^2}{3\varrho\pi^2+1}\sin(\pi x_1 )\sin(\pi x_2)\sin(\pi x_3)
\]
from the state equation. To check the convergence in the $L^2(\Omega)$ and
$H^1(\Omega)$ norms, we first compute the discretization errors
$\|u_\varrho - u_{\varrho h}\|_{L^2(\Omega)}$ and 
$\| \nabla (u_\varrho - u_{\varrho h}) \|_{L^2(\Omega)}$
for finer and finer mesh sizes $h$
and $6$ different values of the regularization parameter $\varrho$;
see Table~\ref{tab:ex1conl2} and Table~\ref{tab:ex1conh1}, respectively. 
Since the given solution $u_\varrho$ is smooth, we observe optimal
convergence rates with respect to the mesh size $h$.
Furthermore, the convergence does not deteriorate with respect to the
regularization parameter $\varrho$.
In Table~\ref{tab:ex1tarcon}, we numerically study the convergence of the
approximation $u_{\varrho h}$ to the target $\bar{u}$ in 
the $L^2(\Omega)$ norm for decreasing $\varrho$ 
by choosing $h=\varrho^{1/2}$ as predicted by Theorem~\ref{thm:errest}. From
Table~\ref{tab:ex1tarcon}, we clearly see a second-order convergence. We
further observe a second-order convergence in the $H^{-1}(\Omega)$ norm; 
see also \cite[Table 5]{NeuSte20}.  

\begin{table}[h]
  \centering
  \begin{tabular}{|r|rr|rr|rr|}
      \hline
      \diagbox{$h$}{$\varrho$}&  $10{^0}$ & eoc  & $10{^{-2}}$ &eoc &  $10{^{-4}}$ & eoc \\ \hline
      $2^{-2}$&$3.1863$e$-3$ &      &$3.7710$e$-2$ &  &$3.9112$e$-2$ &  \\
      $2^{-3}$&$9.0214$e$-4$ &$1.82$ &$8.9461$e$-3$ & $2.08$  &$8.5855$e$-3$ &$2.19$   \\
      $2^{-4}$&$2.2924$e$-4$ &$1.98$ &$2.1970$e$-3$ & $2.03$ &$2.0370$e$-3$ &$2.08$ \\
      $2^{-5}$&$5.7222$e$-5$ &$2.00$ &$5.5283$e$-4$ & $1.99$ &$5.0177$e$-4$ &$2.02$  \\
      $2^{-6}$&$1.4301$e$-5$ &$2.00$ &$1.4011$e$-4$ & $1.98$ &$1.2727$e$-4$&$1.98$\\
      $2^{-7}$&$3.5796$e$-6$ &$2.00$ &$3.5473$e$-5$ & $1.98$ &$3.2775$e$-5$&$1.96$\\
      $2^{-8}$&$8.9633$e$-7$ &$2.00$&$8.9566$e$-6$ &  $1.99$ &$8.4684$e$-6$&$1.95$ \\ \hline 
      \diagbox{$h$}{$\varrho$}&  $10{^{-6}}$ & eoc & $10^{-8}$ & eoc & $10^{-10}$ & eoc\\ \hline
      $2^{-2}$&$3.9223$e$-2$ &  & $3.9224$e$-2$  &  &   $3.9224$e$-2$ & \\
      $2^{-3}$&$8.6079$e$-3$ &$2.19$  &$8.6082$e$-3$  & $2.19$ & $8.6082$e$-3$& $2.19$  \\
      $2^{-4}$&$2.0355$e$-3$ &$2.08$ &$2.0356$e$-3$  &$2.08$  &$2.0356$e$-3$   &$2.08$ \\
      $2^{-5}$&$4.9380$e$-4$ &$2.04$ &$4.9381$e$-4$  &$2.04$  & $4.9381$e$-4$  &$2.04$  \\
      $2^{-6}$&$1.2136$e$-4$  &$2.02$ &$1.2133$e$-4$ & $2.03$ & $1.2133$e$-4$&$2.03$  \\
      $2^{-7}$&$3.0081$e$-5$& $2.01$ &$3.0021$e$-5$ & $2.01$& $3.0021$e$-5$&$2.01$\\
      $2^{-8}$&$7.5374$e$-6$&$2.00$&$7.4604$e$-6$ & $2.01$&$7.4619$e$-6$&$2.01$\\ \hline
  \end{tabular}
  \caption{Example~\ref{subsec:SmoothTarget}: Error $\|u_\varrho - u_{\varrho h}\|_{L^2(\Omega)}$ with respect to $h$ 
  (columns), and fixed, but different $\varrho$ (rows).} 
  \label{tab:ex1conl2}
\end{table}

\begin{table}[h]
  \centering
  \begin{tabular}{|r|rr|rr|rr|}
     \hline
      \diagbox{$h$}{$\varrho$}&  $10{^0}$ & eoc  & $10{^{-2}}$ &eoc&  $10{^{-4}}$ & eoc  \\ \hline
      $2^{-2}$ &$3.2364$e$-2$ &      & $8.2832$e$-1$&  &$1.1721$e$-0$ &\\
      $2^{-3}$ &$1.71297$e$-2$ &$0.92$& $4.1447$e$-1$& $1.00$&$5.5195$e$-1$ &$1.09$  \\
      $2^{-4}$ &$8.6228$e$-3$ &$0.99$ &$2.0490$e$-1$& $1.02$&$2.6817$e$-1$ &$1.04$ \\
      $2^{-5}$ &$4.3068$e$-3$ &$1.00$ & $1.0187$e$-1$& $1.01$&$1.3232$e$-1$ &$1.02$  \\
      $2^{-6}$ &$2.1527$e$-3$ &$1.00$ & $5.0859$e$-2$& $1.00$&$6.5832$e$-2$ &$1.00$\\
      $2^{-7}$ &$1.0771$e$-3$ &$1.00$ & $2.5439$e$-2$& $1.00$&$3.2889$e$-2$&$1.00$\\
      $2^{-8}$ &$5.3905$e$-4$ &$1.00$& $1.2731$e$-2$& $1.00$ &$1.6453$e$-2$&$1.00$ \\ \hline 
      \diagbox{$h$}{$\varrho$}&  $10{^{-6}}$ & eoc & $10^{-8}$ & eoc & $10^{-10}$ & eoc\\ \hline
      $2^{-2}$ &$1.1781$e$-0$ &  & $1.1781$e$-0$ & & $1.1781$e$-0$ &  \\
      $2^{-3}$ &$5.5443$e$-1$ &$1.09$  &$5.5446$e$-1$  & $1.09$ & $5.5446$e$-1$ & $1.09$   \\
      $2^{-4}$ &$2.7032$e$-1$ &$1.04$ &$2.7036$e$-1$  &$1.04$ & $2.7036$e$-1$  & $1.04$ \\
      $2^{-5}$ &$1.3392$e$-1$ &$1.01$ & $1.3397$e$-1$ &$1.01$ & $1.3397$e$-1$  & $1.01$ \\
      $2^{-6}$ &$6.6689$e$-2$ &$1.01$ &$6.6787$e$-2$  &$1.00$ & $6.6789$e$-2$& $1.00$\\
      $2^{-7}$ &$3.3237$e$-2$ &$1.00$ &$3.3367$e$-2$  &$1.00$ & $3.3369$e$-2$& $1.00$\\
      $2^{-8}$ & $1.6566$e$-2$&$1.00$& $1.6683$e$-2$ &$1.00$ &$1.6687$e$-2$&$1.00$ \\  \hline
  \end{tabular}
  \caption{Example~\ref{subsec:SmoothTarget}: Error
    $\| \nabla (u_\varrho - u_{\varrho h}) \|_{L^2(\Omega)}$ with respect to $h$ 
  (columns), and fixed, but different $\varrho$ (rows).} 
  \label{tab:ex1conh1}
\end{table}

\begin{table}[h]
  \centering
  \begin{tabular}{|l|ll|ll|}
      \hline
      $\varrho$ &  $\|\bar{u} - u_{\varrho h}\|_{H^{-1}(\Omega)}^2$ & eoc &  $\|\bar{u} - u_{\varrho h}\|_{L^2(\Omega)}^2$ & eoc \\ \hline
      $10{^0}$ & $3.54886$e$-3$ & & $1.16966$e$-1$ &  \\      
      $10{^{-1}}$ &$2.11939$e$-3$  &$0.22$  &$6.98529$e$-2$ &$0.22$   \\
      $10{^{-2}}$ & $1.97952$e$-4$ & $1.03$ & $6.52427$e$-3$ &$1.03$  \\
      $10{^{-3}}$ & $3.13686$e$-6$ & $1.80$ & $1.03388$e$-4$ & $1.80$  \\
      $10{^{-4}}$ & $3.30578$e$-8$ & $1.98$ & $1.08962$e$-6$ & $1.98$ \\
      \hline
  \end{tabular}
  \caption{
    Example~\ref{subsec:SmoothTarget}: 
    Errors $\|\bar{u} - u_{\varrho h}\|_{H^{-1}(\Omega)}^2$ and
    $\|\bar{u} - u_{\varrho h}\|_{L^2(\Omega)}^2$ \newline for
    $\bar{u}\in H^1_0(\Omega) \cap H^2(\Omega)$.}
  \label{tab:ex1tarcon}
\end{table}

\noindent
Table~\ref{tab:ex1amgitertime} 
provides the number of AMG preconditioned
CG iterations, and the  total coarsening and solving time measured in
second (s) with respect to $h$ and $\varrho$. From these numerical results, it
is clear to see the relative robustness of the AMG preconditioner with respect
to the mesh refinement and the decreasing regularization parameter
$\varrho$. The computational time scales well with respect to the number of
unknowns. In Table~\ref{tab:ex1bddcittime}, we study the robustness with
respect to $\varrho$ and parallel performance (strong scaling) of the
BDDC preconditioned CG solver. So, we fix the total number of unknowns
to  $\#Dofs=2,146,689$ that is related to $h=1/128$. We observe that
the numbers of BDDC iterations are rather stable with respect to the
number of subdomains $p$, and the varying regularization parameter $\varrho$. 
Further, we observe perfect strong scalability (checking each row). 
Over-scaling from $32$ to $64$ cores may result from heterogeneity
of the cluster, larger memory consumption for small number of cores, or the
shared node resources with other users. 

\begin{table}[h]
  \centering
  {\footnotesize
    \begin{tabular}{|r|rrrrrr|}
      \hline
      \diagbox{$\varrho$}{$h$}&   $\frac{1}{8}$ & $\frac{1}{16}$ & $\frac{1}{32}$ & $\frac{1}{64}$ &  $\frac{1}{128}$ & $\frac{1}{256}$\\ \hline
      $10{^0}$&8\,(0.01 s)&10\,(0.06 s)&12\,(0.72 s)&14\,(8.29 s)&17\,(127.06 s)&20\,(1420.73 s) \\
      $10{^{-2}}$&6\,(0.01 s)&8\,(0.05 s)&9\,(0.56 s)&11\,(6.83 s)&13\,(94.28 s)&15\,(910.77 s) \\
      $10{^{-4}}$&5\,(0.00 s)&4\,(0.04 s)&4\,(0.38 s)&5\,(4.51 s)&7\,(65.85 s)&8\,(646.14 s)\\
      $10{^{-6}}$&5\,(0.00 s)&5\,(0.04 s)&5\,(0.41 s)&5\,(4.46 s)&4\,(52.90 s)&4\,(499.10 s)\\
      $10{^{-8}}$&5\,(0.00 s)&5\,(0.04 s)&5\,(0.41 s)&5\,(4.46 s)&5\,(57.59 s)&5\,(530.97 s)\\  
      $10{^{-10}}$&5\,(0.00 s)&5\,(0.04 s)&5\,(0.41 s)&5\,(4.43 s)&5\,(57.04 s)&5\,(560.33 s)\\ 
      $10{^{-12}}$&5\,(0.00 s)&5\,(0.04 s)&5\,(0.41 s)&5\,(4.42 s)&5\,(54.63 s)&5\,(603.02 s)   \\
      \hline
  \end{tabular}
  \caption{Example~\ref{subsec:SmoothTarget}: 
    Number of AMG preconditioned CG iterations, total
    coarsening and solving time measured in second (s), with respect to $h$ and $\varrho$.}
    \label{tab:ex1amgitertime}
    }
\end{table}

\begin{table}[h]
  \centering
  \begin{tabular}{|r|rrrr|}
      \hline
      \diagbox{$\varrho$}{$p$}&  $32$ & $64$ & $128$ & $256$  \\ \hline
      $10{^0}$&34\,(46.43 s)&34\,(16.04 s)&35\,(5.92 s)&36\,(2.98 s)\\
      $10{^{-2}}$&32\,(43.84 s)&32\,(15.20 s)&33\,(5.58 s)&34\,(2.70 s)\\
      $10{^{-4}}$&27\,(37.20 s)&27\,(13.12 s)&28\,(4.85 s)&30\,(2.39 s)\\   
      $10{^{-6}}$&20\,(28.01 s)&20\,(9.77 s)&22\,(3.83 s) &22\,(1.87 s)\\ 
      $10{^{-8}}$&20\,(28.15 s)&20\,(9.75 s)&21\,(3.65 s)&21\,(1.76 s)\\  
      $10{^{-10}}$&20\,(28.02 s)&20\,(9.69 s)&21\,(3.63 s)&21\,(1.74 s)\\
      $10{^{-12}}$&20\,(28.01 s)&20\,(9.71 s)&21\,(3.69 s)&21\,(1.67 s)\\
      \hline
  \end{tabular}
  \caption{Example~\ref{subsec:SmoothTarget}: Number of iterations, and time
    measured in second (s) for the BDDC preconditioned CG solver, where
    $h=1/128$ and $\#Dofs=2,146,689$.}
  \label{tab:ex1bddcittime} 
\end{table}

\subsection{Discontinuous target}
\label{subsec:DiscontinuousTarget}

In the second example, we consider a discontinuous target function
\begin{equation}
  \bar{u}(x)=
  \begin{cases}
    1 \;\textup { for }  x\in (0.25, 0.75)^3\\
    0 \;\textup { for }  x\in \Omega\setminus (0.25, 0.75)^3,\\
  \end{cases}
\end{equation}
that is similar to the first example in \cite{NeuSte20}. In this case, we
solve the equation on adaptively refined meshes that are driven by a
residual based error indicator \cite{RVerfurth2013}. 
This leads to a lot of local refinements near the interface where the target
is discontinuous. The energy regularization leads to a control that
concentrates on the interface. As an illustration, we visualize
the adaptive mesh, the state $u_{\varrho h}$, and the control $z_{\varrho h}$ 
in Fig.~\ref{fig:distarfunc}, which are comparable to the results in
\cite[Fig.~3]{NeuSte20}. Moreover, we provide the convergence of the
approximation $u_{\varrho h}$
to the target $\bar{u}$ with respect to the regularization
parameter $\varrho$ in Table~\ref{tab:ex2tarcon}, i.e., an
order $\sqrt{1.5}$ for 
the approximation in $H^{-1}(\Omega)$ and $\sqrt{0.5}$ in $L^2(\Omega)$. 
In this case, since $\bar{u}\in H^{1/2-\varepsilon}(\Omega)$, we obtain
reduced convergence of the finite element approximation to the target; 
see also \cite[Table~1]{NeuSte20}. Due to adaptive refinements for different
regularization parameters $\varrho=10^k$, $k=0,-1,\ldots,-6$, the final
meshes may contain different number of degrees of freedom. In
Fig.~\ref{table:ex2amg}, we visualize the number of AMG preconditioned CG
iterations (left plot) and the computational time in seconds (right plot)
for different values of $\varrho$ during the adaptive refinement procedure. 
We observe that the AMG preconditioned CG iterations stay in a small range
between $5$ and $25$; the computational time scales well with respect to
the number of unknowns. To compare the performance of BDDC
preconditioners with respect to an adaptive refinement, we select
similar numbers of Dofs for all regularization parameters,
i.e., \#Dofs=2,210,254 ($\varrho=10^0$), \#Dofs=2,278,661 ($\varrho=10^{-1}$),
\#Dofs=2,537,773 ($\varrho=10^{-2}$), \#Dofs=1,853,354 ($\varrho=10^{-3}$),
\#Dofs=1,301,825 ($\varrho=10^{-4}$), \#Dofs=1,910,829 ($\varrho=10^{-5}$),
and \#Dofs=1,895,056 ($\varrho=10^{-6}$). In Table~\ref{tab:ex2bddctimeit},
we present the number of BDDC preconditioned CG iterations, and the
corresponding solving time, measured in second (s), with respect to the
number of subdomains $p$ and the regularization parameter $\varrho$. For
such adaptive meshes, we observe relatively good performance in both the
number of BDDC preconditioned CG iterations and the computational time,
as well as good strong scalability with respect to the number of
subdomains $p$.

\begin{figure}[h]
  \centering
  \includegraphics[width=0.34\textwidth]{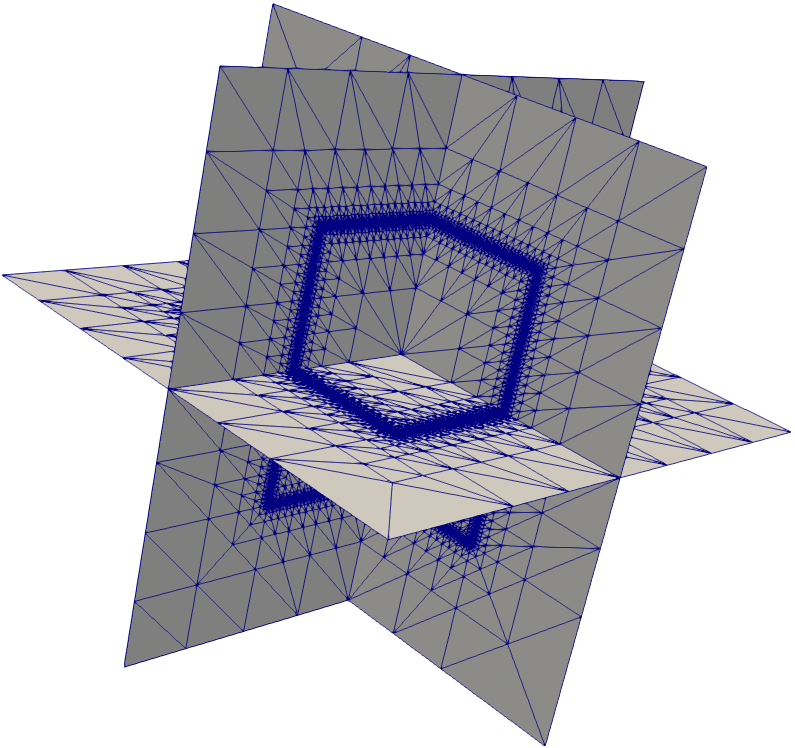}
  \includegraphics[width=0.34\textwidth]{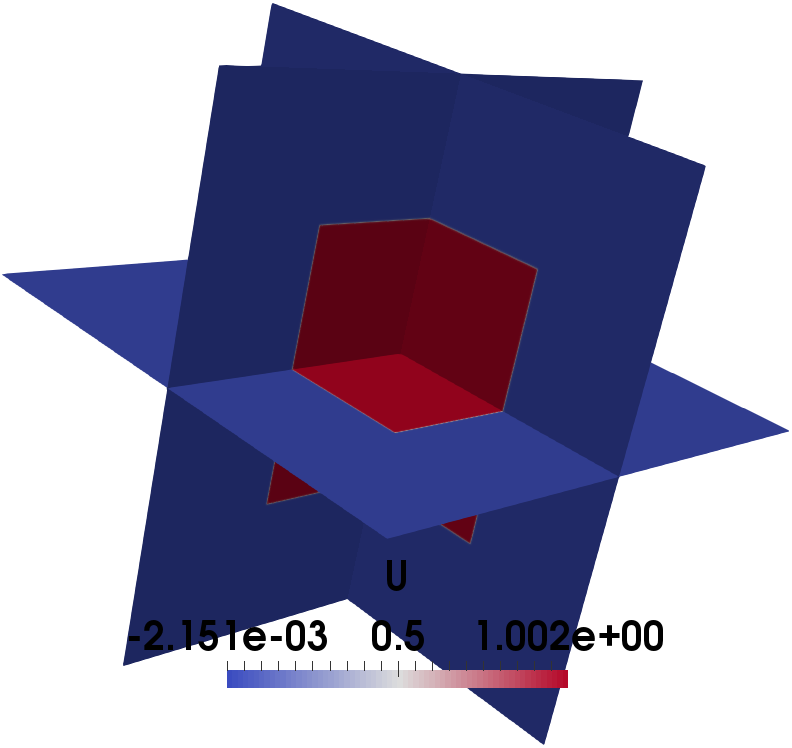}
  \includegraphics[width=0.28\textwidth]{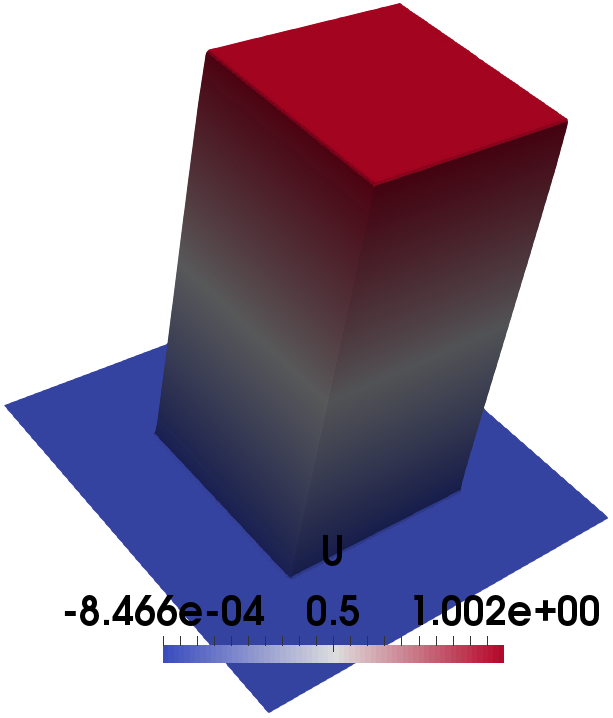}
  \includegraphics[width=0.34\textwidth]{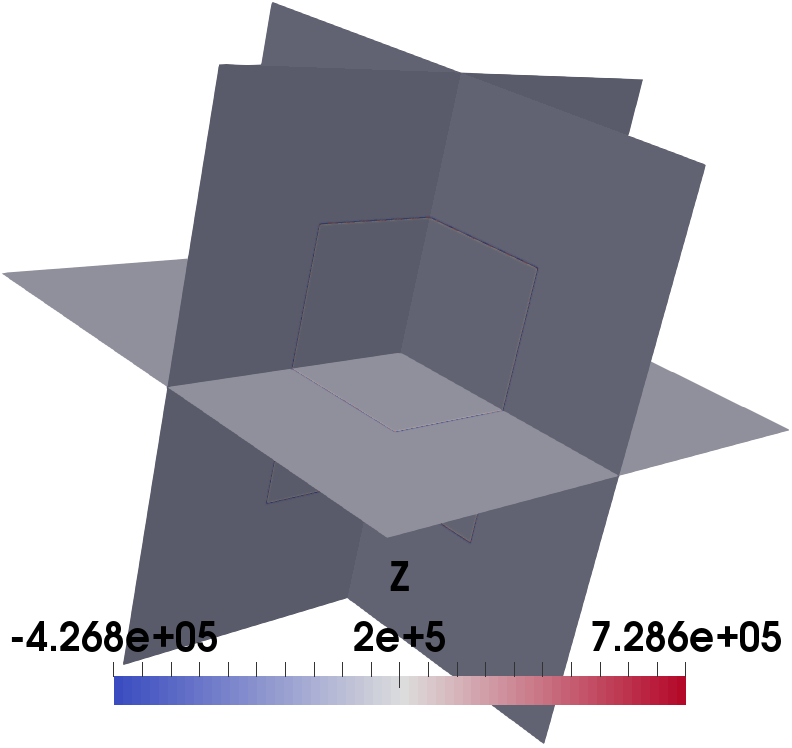}
  \includegraphics[width=0.29\textwidth]{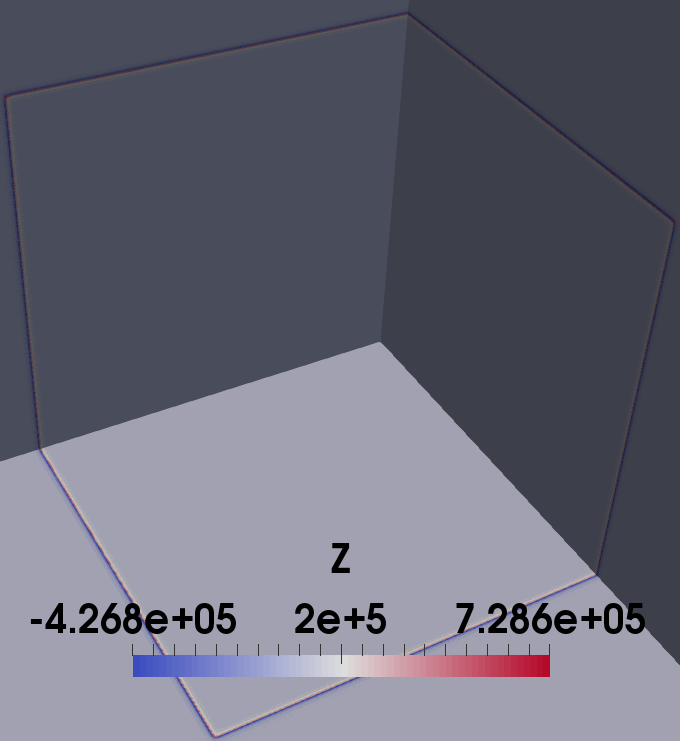}
  \includegraphics[width=0.32\textwidth]{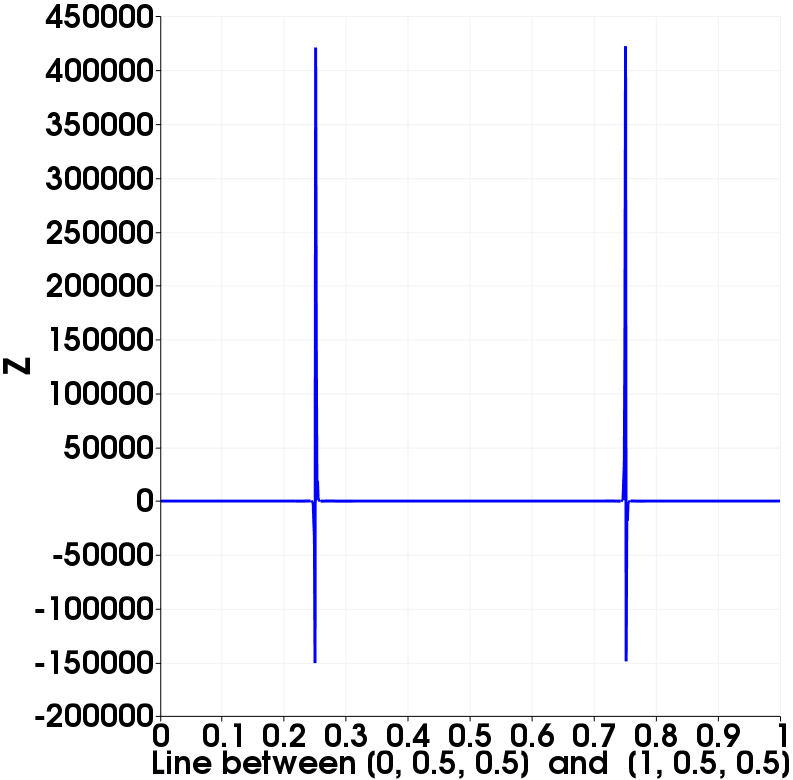}  
  \caption{Example~\ref{subsec:DiscontinuousTarget}: Adaptive meshes at the $21$th refinement step with $1,895,056$
    Dofs, $u_{\varrho h}$ in $\Omega$ and on the cutting plane $x_3=0.5$ (up), 
    control in $\Omega$, near the interface, and on the line between $[0, 0.5, 0.5]$ and $[1, 0.5, 0.5]$ (down),
    $\varrho=10^{-6}$.}   
  \label{fig:distarfunc}
\end{figure} 

\begin{table}[h]
  \centering
  \begin{tabular}{|r|ll|ll|}
      \hline
      $\varrho$ &  $\|\bar{u} - u_{\varrho h}\|_{H^{-1}(\Omega)}^2$ & eoc & $\|\bar{u} - u_{\varrho h}\|_{L^2(\Omega)}^2$ & eoc  \\ \hline
      $10{^0}$  &$2.1308$e$-3$ & & $1.2018$e$-1$  & \\
      $10{^{-1}}$&$1.3401$e$-3$ & $0.20$ & $9.0747$e$-2$ & $0.12$  \\
      $10{^{-2}}$&$1.9101$e$-4$&$0.85$&$3.6563$e$-2$&$0.39$  \\ 
      $10{^{-3}}$&$9.0010$e$-6$&$1.33$&$1.1894$e$-2$&$0.49$  \\
      $10{^{-4}}$&$3.1821$e$-7$&$1.45$&$3.8028$e$-3$& $0.50$  \\
      $10{^{-5}}$&$1.0474$e$-8$  & $1.48$ & $1.2075$e$-3$ & $0.50$ \\   
      $10{^{-6}}$&$3.3491$e$-10$ & $1.50$ & $3.7988$e$-4$ & $0.50$ \\     
      \hline
  \end{tabular}
  \caption{Example~\ref{subsec:DiscontinuousTarget}: Errors $\|\bar{u} - u_{\varrho h}\|_{H^{-1}(\Omega)}^2$ 
    and $\|\bar{u} - u_{\varrho h}\|_{L^2(\Omega)}^2$ of the approximations $u_{\varrho h}$ 
    to the target $\bar{u}\in H^{1/2-\varepsilon}(\Omega)$.} 
  \label{tab:ex2tarcon}
\end{table}

\begin{figure}[h]
  \centering
  \includegraphics[width=0.49\textwidth]{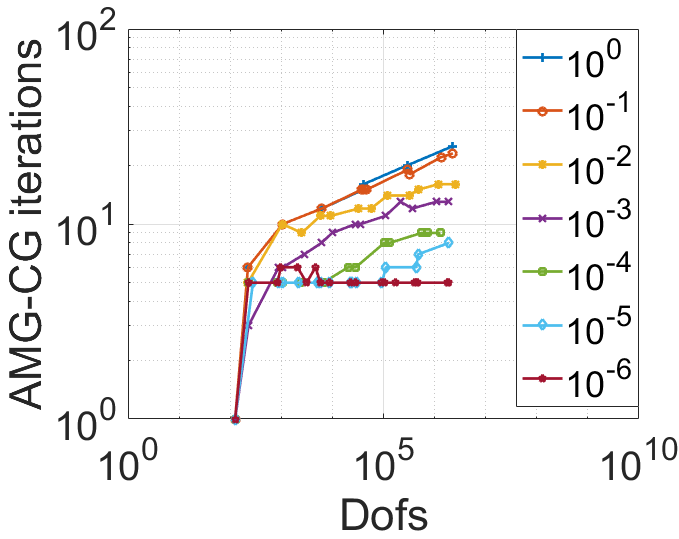}
  \includegraphics[width=0.49\textwidth]{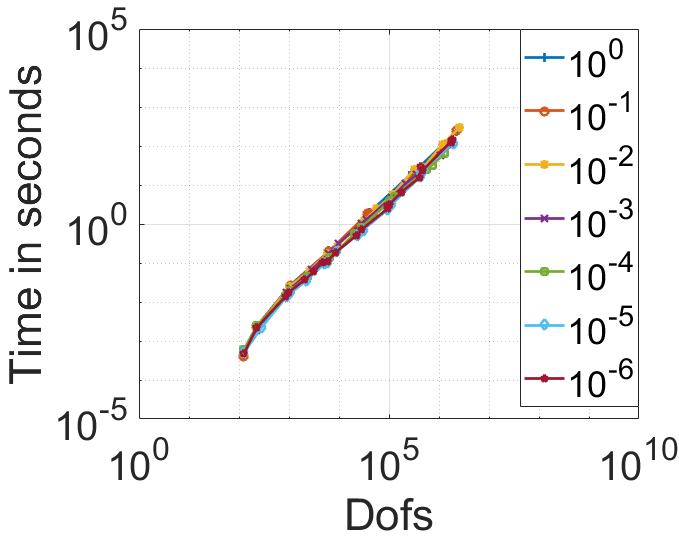}
  \caption{Example~\ref{subsec:DiscontinuousTarget}: Number of AMG preconditioned CG iterations (left) and
    computational time in seconds (right) with respect to different
    $\varrho\in\{10^0, 10^{-1},...10^{-6}\}$.}  
  \label{table:ex2amg}
\end{figure} 
\begin{table}[h]
  \centering
  \begin{tabular}{|r|rrrr|r|}
      \hline
       \diagbox{$\varrho$}{$p$}&  $32$ & $64$ & $128$ & $256$ & \#Dofs \\ \hline
       $10{^0}$ &37 (60.77 s) &46 (24.34 s) & 42 (8.20 s) & 52 (5.02 s)  & $2,210,254$ \\
       $10^{-1}$ &42 (66.85 s) &45 (24.85 s) & 46 (9.28 s) & 43 (4.64 s)  & $2,278,661$ \\
       $10^{-2}$ &34 (60.06 s) &38 (23.86 s) & 42 (9.75 s) & 43 (4.79 s) &  $2,537,773$ \\
       $10^{-3}$ &35 (40.92 s) &39 (15.35 s) & 35 (5.12 s) & 40 (3.23 s) & $1,853,354$ \\   
       $10^{-4}$ &33 (16.79 s) &31 (6.30 s) & 39 (3.17 s) & 36 (1.53 s) & $1,301,825$ \\  
       $10^{-5}$ &33 (17.63 s) &33 (7.65 s) & 33 (3.28 s) & 38 (1.88 s)  & $1,910,829$\\ 
       $10^{-6}$ &36 (18.30 s) &37 (8.42 s) & 35 (3.31 s) & 44 (1.93 s) &  $1,895,056$\\
      \hline
  \end{tabular}
  \caption{Example~\ref{subsec:DiscontinuousTarget}: Number of BDDC preconditioned CG
    iterations, solving time measured in second (s) with respect to the number
    of subdomains $p$ (= number of cores), regularization parameter $\varrho$,
    and different adaptive meshes (\#Dofs).}  
  \label{tab:ex2bddctimeit} 
\end{table}

\subsection{Smooth target with non-zero boundary conditions}
\label{subsec:SmoothTargetViolatBND}
In the third example, we use the smooth target 
\[
\bar{u}(x_1,x_2,x_3)=1+\sin(\pi x_1)\sin(\pi x_2)\sin(\pi x_3)
\]
considered in \cite{NeuSte20}. 
This target  violates the homogeneous Dirichlet boundary conditions. 
As in the previous example, we solve the equation on  
adaptively refined meshes, which results in many local refinements near the
boundary, where the control concentrates. For an illustration, we visualize
the adaptive mesh, the state $u_{\varrho h}$, and the control $z_{\varrho h}$ in
Fig.~\ref{fig:hstarfunc}. Further, we observe the convergence order
$\sqrt{1.5}$ for the approximation in $H^{-1}(\Omega)$, and
$\sqrt{0.5}$ in $L^2(\Omega)$ with respect to $\varrho$; see
Table~\ref{tab:ex3tarcon}. Since $\bar{u}\in C^\infty(\overline{\Omega})$,
but $\bar{u}\notin H_0^1(\Omega)$, we only expect
a reduced order of convergence for this example; 
see also \cite[Table 7]{NeuSte20}. In Fig. \ref{fig:ex3amg}, 
for each $\varrho$, we visualize the number
of AMG preconditioned CG iterations (left plot) and the computational time in
seconds (right plot) during the adaptive refinement levels. 
The AMG preconditioned CG iterations are in a small range between $1$ and $20$. 
The computational time scales well with respect to the number of unknowns. 
To see the performance of the BDDC  preconditioner with respect to an adaptive
refinement, we select similar numbers of Dofs for all regularization
parameters, i.e., \#Dofs=2,146,491 ($\varrho=10^0$),
\#Dofs=2,146,689 ($\varrho=10^{-1}$),
\#Dofs=2,146,575 ($\varrho=10^{-2}$), \#Dofs=2,654,801 ($\varrho=10^{-3}$),
\#Dofs=1,997,688 ($\varrho=10^{-4}$), \#Dofs=3,688,105 ($\varrho=10^{-5}$),
and \#Dofs=3,676,447 ($\varrho=10^{-6}$). 
Table~\ref{tab:ex3bddctimeit} presents
the number of BDDC preconditioned CG iterations, 
and solving times measured in second (s) 
with respect to the number of subdomains $p$ and
the regularization parameter $\varrho$. We again observe a
relatively good performance in both the number of BDDC preconditioned CG
iterations and the computational time, as well as good strong
scalability with respect to the number of subdomains $p$. 

\begin{figure}[h]
  \centering
  \includegraphics[width=0.325\textwidth]{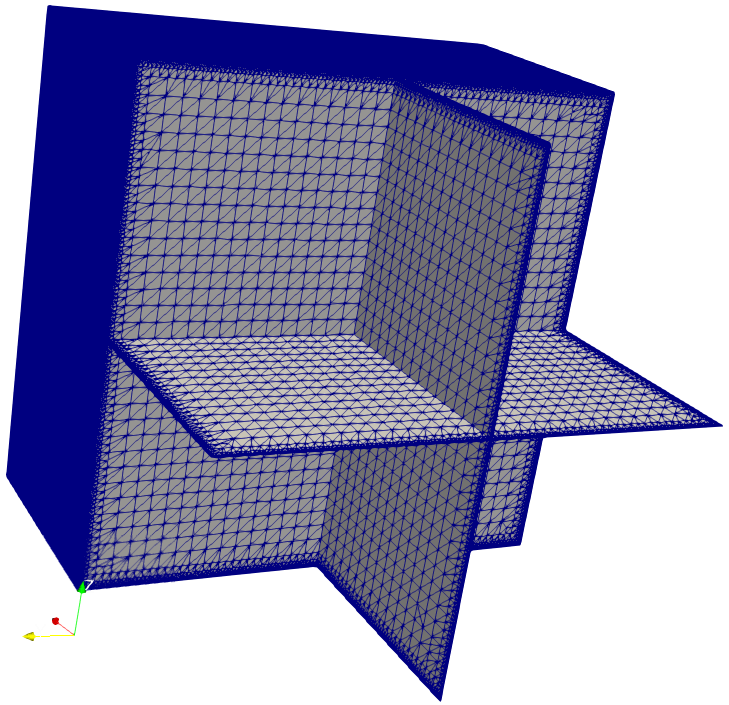}
  \includegraphics[width=0.325\textwidth]{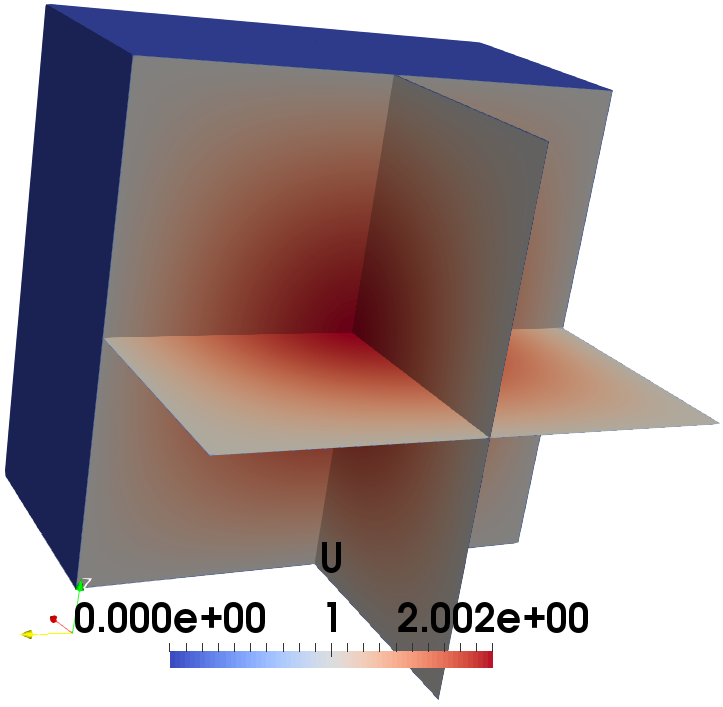}
  \includegraphics[width=0.325\textwidth]{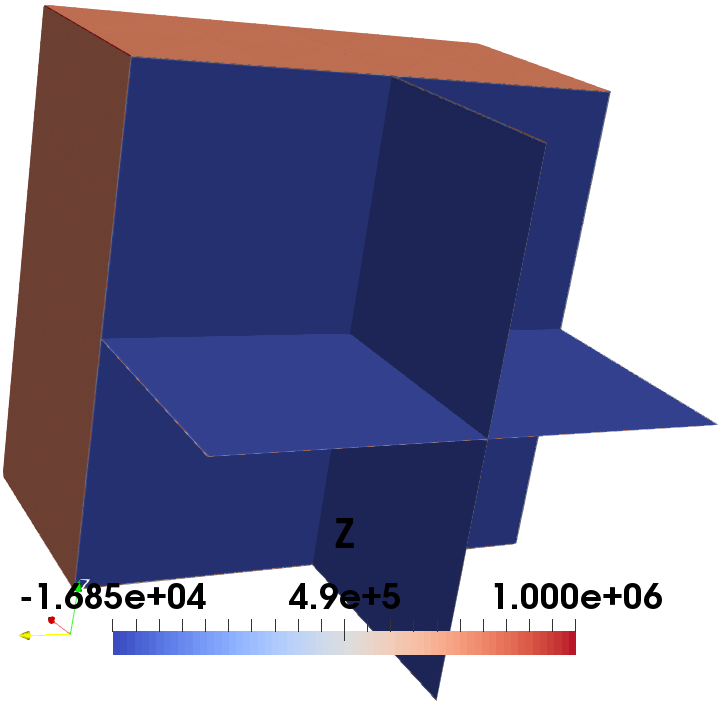}
  \includegraphics[width=0.32\textwidth]{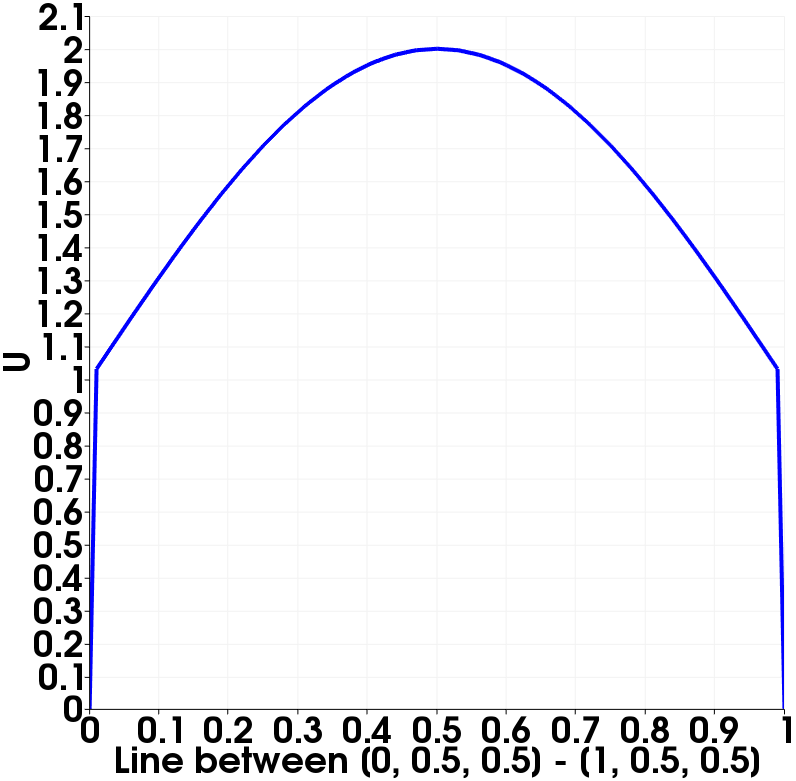}  
  \includegraphics[width=0.35\textwidth]{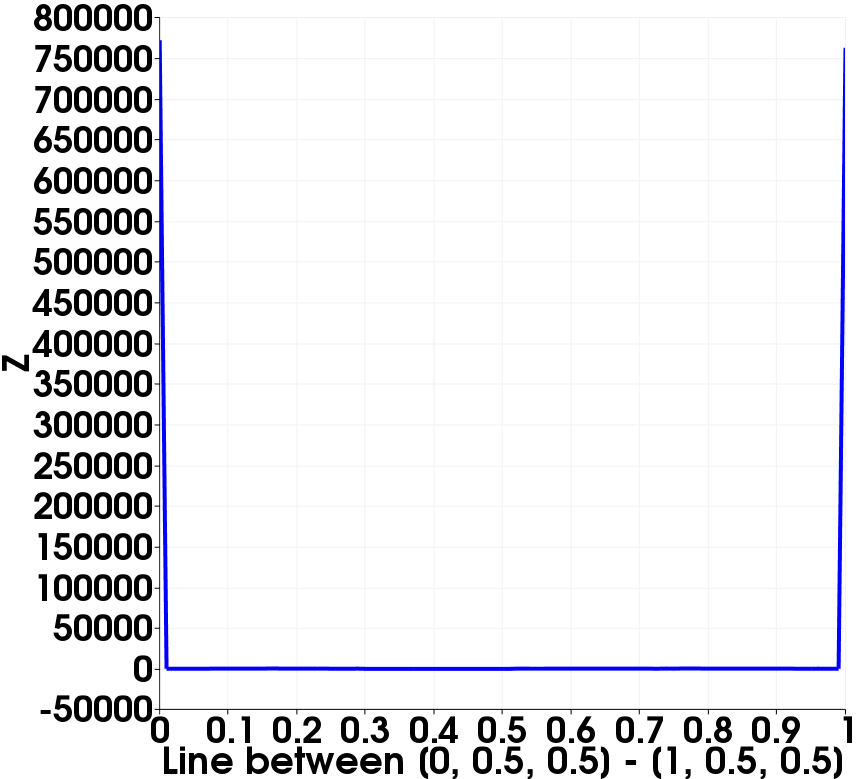}
  \caption{Example~\ref{subsec:SmoothTargetViolatBND}: Adaptive meshes at the
    $12$th refinement step with $3,688,306$ Dofs, $u_{\varrho h}$ in $\Omega$ and $z_{\varrho h}$
    in $\Omega$ (up); $u_{\varrho h}$ and $z_{\varrho h}$ along the line between $[0, 0.5, 0.5]$ and
    $[1, 0.5, 0.5]$ (down), $\varrho=10^{-6}$.}   
  \label{fig:hstarfunc}
\end{figure} 

\begin{table}[h]
  \centering
  \begin{tabular}{|r|ll|ll|}
      \hline
      $\varrho$ &  $\|\bar{u} - u_{\varrho h}\|_{H^{-1}(\Omega)}^2$ & eoc & $\|\bar{u} - u_{\varrho h}\|_{L^2(\Omega)}^2$ & eoc  \\ \hline
      $10{^0}$& $3.5241$e$-2$ & & $1.5613$e$-0$  & \\
      $10{^{-1}}$&$2.1590$e$-2$ & $0.21$ & $1.0842$e$-0$ & $0.16$\\
      $10{^{-2}}$&$2.5430$e$-3$&$0.93$&$3.2593$e$-1$&$0.52$\\ 
      $10{^{-3}}$&$8.8234$e$-5$&$1.46$&$9.5243$e$-2$&$0.53$\\
      $10{^{-4}}$&$2.7302$e$-6$&$1.51$&$3.0292$e$-2$&$0.50$\\
      $10{^{-5}}$&$8.5525$e$-8$&$1.50$&$9.6709$e$-3$&$0.50$\\   
      \hline
  \end{tabular}
  \caption{Example~\ref{subsec:SmoothTargetViolatBND}: Errors $\|\bar{u} - u_{\varrho h}\|_{H^{-1}(\Omega)}^2$ 
    and $\|\bar{u} - u_{\varrho h}\|_{L^2(\Omega)}^2$ of the approximations $u_{\varrho h}$ 
    to the target $\bar{u}\in C^\infty(\overline{\Omega}$, $\bar{u}\notin H_0^1(\Omega)$.} 
  \label{tab:ex3tarcon}
\end{table}

\begin{figure}[h]
  \centering
  \includegraphics[width=0.49\textwidth]{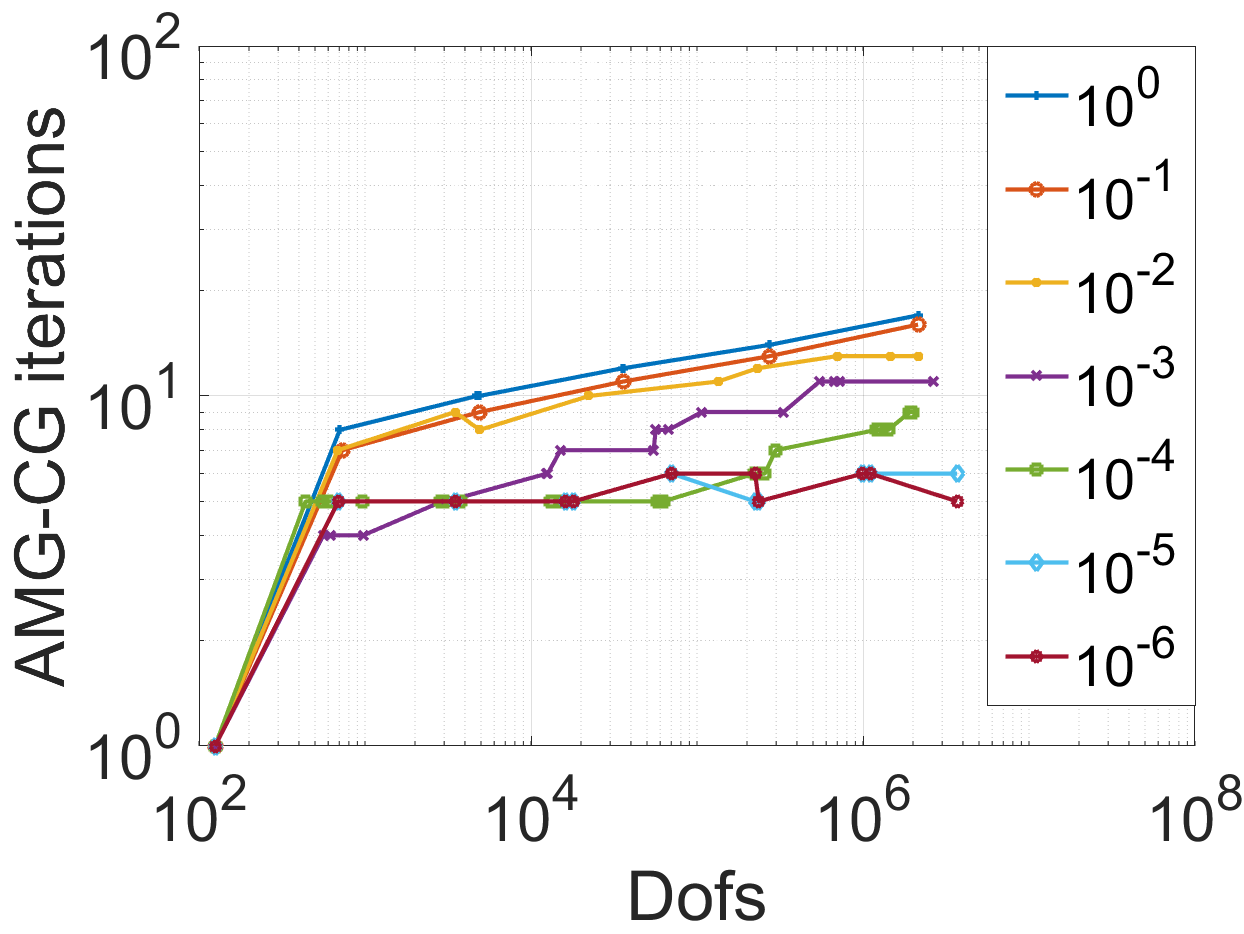}
  \includegraphics[width=0.49\textwidth]{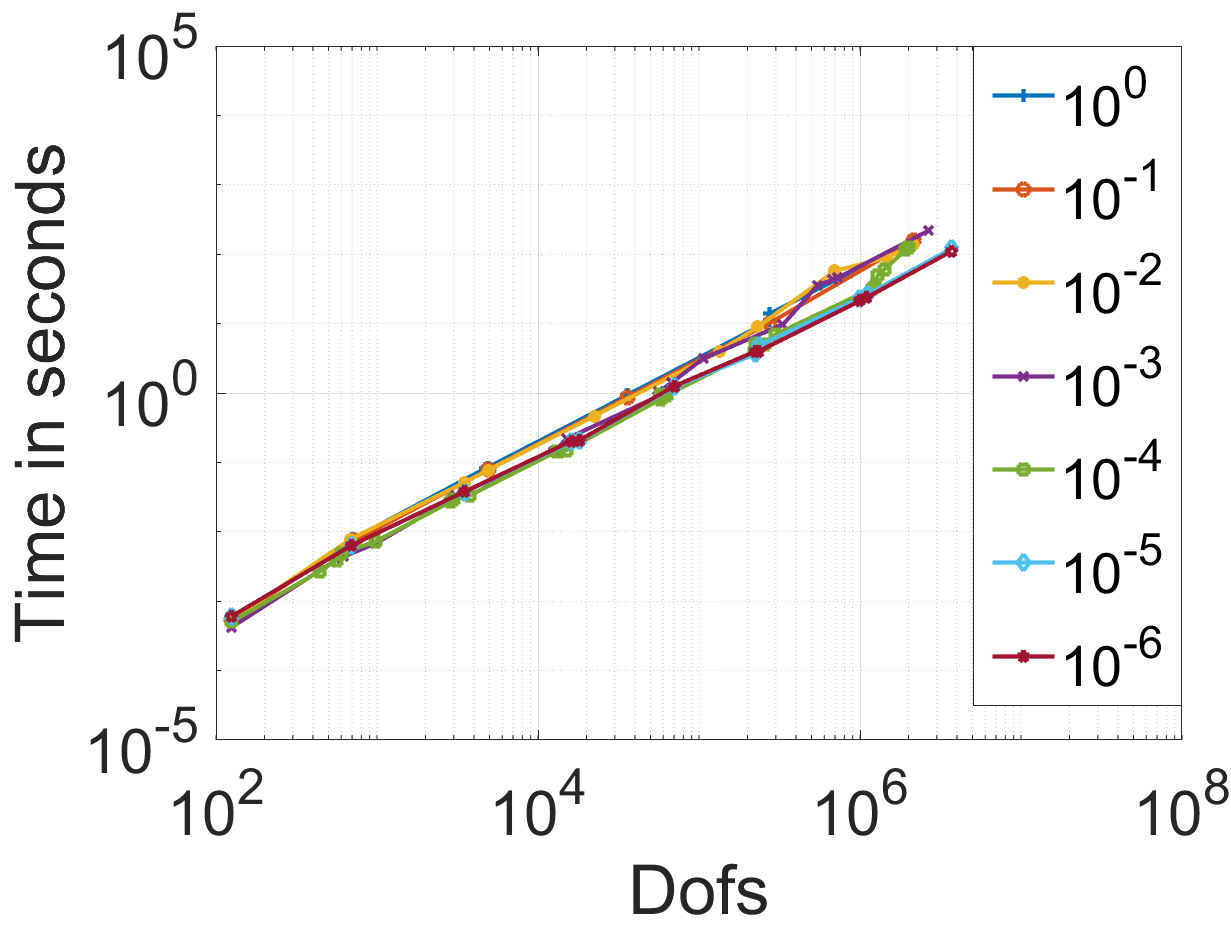}
  \caption{Example~\ref{subsec:SmoothTargetViolatBND}: Number of AMG preconditioned CG iterations (left) and
    computational time in seconds (right) with respect to different
\newline   $\varrho\in\{10^0, 10^{-1},...10^{-6}\}$.}  
  \label{fig:ex3amg}
\end{figure} 

\begin{table}[h]
  \centering
  \begin{tabular}{|r|rrrr|r|}
      \hline
       \diagbox{$\varrho$}{$p$}&  $32$ & $64$ & $128$ & $256$ & \#Dofs \\ \hline
       $10{^0}$&33 (45.61 s)&32 (15.37 s)&32 (5.62 s)&33 (2.78 s)&$2,146,491$\\
       $10^{-1}$&30 (41.51 s)&31 (15.16 s)&32 (5.71 s)&32 (2.84 s)&$2,146,689$\\ 
       $10^{-2}$&30 (42.13 s)&31 (15.46 s)&30 (5.14 s)&32 (2.59 s)&$2,146,575$\\ 
       $10^{-3}$&29 (32.46 s)&34 (14.30 s)&32 (7.20 s)&39 (3.27 s)&$2,654,801$\\   
       $10^{-4}$&29 (15.87 s)&31 (6.98 s)&31 (2.94 s)&29 (1.40 s)&$1,997,688$\\  
       $10^{-5}$&30 (21.03 s)&30 (9.22 s)&33 (4.34 s)&37 (2.73 s)&$3,688,105$\\ 
       $10^{-6}$&31 (22.57 s)&33 (10.06 s)&31 (4.05 s)&38 (2.57 s)&$3,676,447$\\
      \hline
  \end{tabular}
  \caption{Example~\ref{subsec:SmoothTargetViolatBND}: Number of BDDC preconditioned CG
    iterations, solving time measured in second (s)
    with respect to the number of subdomains $p$ (= number of cores), regularization
    parameter $\varrho$, and different adaptive meshes (\#Dofs).} 
  \label{tab:ex3bddctimeit} 
\end{table}

%
%
\section{Conclusions}\label{sec:con}
Using estimates for the error between the exact solution $u_\varrho$ 
of the optimal control problem \eqref{eq:optcon}
and the target $\overline{u}$ in terms of the regularization 
parameter $\varrho$, derived in \cite{NeuSte20}, we have investigated 
the error between the finite element solution $u_{\varrho h}$ 
and the target $\overline{u}$ in terms of $\varrho$ and $h$.
Furthermore, we have studied AMG and BDDC preconditioned CG solvers 
for (adaptive) finite element equations 
arising from the elliptic optimal control problem with energy regularization. 
Both preconditioners have shown good performance with
respect to (adaptive) mesh refinements and a (decreasing) regularization
parameter. Moreover, the BDDC preconditioner has shown its strong
scalability with
respect to the number of subdomains on a distributed memory computer.    

While in this paper we have considered a distributed control problem
subject to the Poisson
equation, a related analysis can be done in the case of a parabolic
heat equation. These results will be published elsewhere, but
see \cite{LSTY:Energy} for a space-time discretization approach.

\section*{Acknowledgments}
The third author was partly supported by the Austrian Science Fund (FWF)
via the grant NFN S117-03.


\bibliography{precondreactiondiff}
\bibliographystyle{abbrv}

\end{document}